\documentclass[14pt]{article}
\usepackage{amsmath}
\usepackage{amssymb}
\usepackage{mathrsfs}
\usepackage{amscd}
\usepackage{amsthm,color}
\usepackage[matrix,frame,import,curve]{xypic}
\usepackage{authblk}

\newtheorem{lemma}{Lemma}[section]
\newtheorem{theorem}[lemma]{Theorem}
\newtheorem{corollary}[lemma]{Corollary}
\newtheorem{conjecture}[lemma]{Conjecture}
\newtheorem{prop}[lemma]{Proposition}
\theoremstyle{definition}
\newtheorem{definition}[lemma]{Definition}

\newtheorem{remark}[lemma]{Remark}

\theoremstyle{remark}
\newtheorem*{proof*}{Proof}
\numberwithin{equation}{section}

\def\Spec{{\bf {Spec}}}
\def\RHom{{\mathbf{R}\rm{Hom}}}

\def\D{\mathrm{D}}

\def\Ext{{\mathrm{Ext}}}
\def\sExt{{\mathscr{E}xt}}
\def\Hom{{\mathrm{Hom}}}
\def\sHom{{\mathscr{H}om}}
\def\End{{\mathrm{End}}}
\def\Tor{{\mathrm{Tor}}}

\def\Spec{{\mathrm{Spec\ }}}
\def\deg{{\mathrm{deg}}}
\def\perf{{\mathfrak{Perf}}}

\def\ZZ{{\mathbb Z}}
\def\CC{{\mathbb C}}

\def\AA{{\mathbb A}}

\def\HH{{\mathrm{HH}}}

\def\bR{{\bf{R}}}
\def\bL{{\bf{L}}}

\def\cG{{\cal{G}}}
\def\cF{{\cal{F}}}

\def\cO{{\cal{O}}}
\def\cC{{\cal{C}}}
\def\cL{{\cal{L}}}
\def\cM{{\cal{M}}}
\def\cN{{\cal{N}}}

\def\cH{{\cal{H}}}
\def\cP{{\cal{P}}}
\def\cT{{\cal{T}}}
\def\cU{{\cal{U}}}

\def\cR{{\cal{R}}}
\def\cD{{\cal{D}}}
\def\cX{{\cal{X}}}
\def\cY{{\cal{Y}}}

\def\tr{{\mathrm{tr}}}
\def\Res{{\mathrm{Res}}}
\def\deg{{\mathrm{deg}}}

\def\ep{{\epsilon}}

\title{Contraction algebra and invariants of singularities}
\date{}
\author[1]{Zheng Hua\thanks{huazheng@maths.hku.hk}}
\author[2]{Yukinobu Toda \thanks{yukinobu.toda@ipmu.jp}}
\affil[1]{Department of Mathematics, the University of Hong Kong}
\affil[2]{Kavli Institute for the Physics and Mathematics of the Universe, University of Tokyo}

\begin{document}
\maketitle
\begin{abstract}
In \cite{DW13}, Donovan and Wemyss introduced the contraction algebra of flopping curves in 3-folds. When the flopping curve is smooth and irreducible, we prove that the contraction algebra together with its $A_\infty$-structure recovers various invariants associated to the underlying singularity and its small resolution, including the derived category of singularities and the genus zero Gopakumar-Vafa invariants.
\end{abstract}

\section{Introduction}\label{sec:intro}
\subsection{Motivation}
In their recent works~\cite{DW13}, \cite{DW15}, 
Donovan and Wemyss constructed certain algebras called 
\emph{contraction algebras} associated with birational morphisms $f:X\to Y$ with at most one-dimensional fibers. The contraction algebras prorepresent the functors of noncommutative deformations of the reduced fiber of $f$. 
When $f$ is a 3-fold flopping contraction, 
they described the Bridgeland-Chen's 
flop-flop autoequivalence
of the derived category of coherent sheaves on $X$
in terms of the twisted functor associated with the 
contraction algebra. 
This is a remarkable application of non-commutative deformation 
theory first developed by Laudal to the commutative algebraic geometry, 
and provides a deep understanding of 
floppable $(1, -3)$-curves, 
mysterious curves in birational geometry.  
Indeed they  
conjectured that 
their contraction algebra
recovers the formal neighborhood at the 
flopping curve, so that the contraction 
algebra contains enough information for the local 
geometry of the flopping curve. 
The above Donovan-Wemyss's conjecture is one of the 
motivations in this article.  

The Donovan-Wemyss's contraction algebra is also interesting in 
enumerative geometry. 
As an analogy of Reid's width for a $(0, -2)$-curve, 
they introduced the new invariant for a $(1 -3)$-curve
called the non-commutative width 
as the dimension of the contraction algebra. 
In~\cite{To14},
the second author showed that 
 their non-commutative widths are
described by
weighted sums of 
 Katz's genus zero Gopakumar-Vafa (GV) invariants, 
which virtually count rational curves on Calabi-Yau 3-folds. 
The formula proved in~\cite{To14} (cf.~Theorem~\ref{cor:dim}) 
is 
quite simple, but 
the proof of~\cite{To14} was not 
enough intrinsic to explain its meaning. 
Also it was not shown whether the contraction algebra 
recovers the GV invariants 
of all degrees or not. 
This article is also 
motivated by the above questions 
on the relationship between contraction algebras and 
GV invariants. 

\subsection{Results and the organization of the paper}
After reviewing some 
basic terminology in Section~\ref{sec:pre}, 
we study 
an alternative description 
of contraction algebras in Section~\ref{sec:struc}. 
We will show that the 
contraction algebras
are described in terms of 
(derived) endomorphism algebras 
in Orlov's derived category of singularities
$\D_{sg}(Y)$ of $Y$
(Corollary~\ref{Ext1v}). 
This result will be used 
to provide 
some additional structures on contraction 
algebras. 
We see that our alternative description of the contraction algebra
turns it into a $\ZZ/2$-graded $A_\infty$-algebra with a Frobenius structure, which we call the $\infty$-contraction algebra. It reduces to 
Donovan-Wemyss's contraction algebra by forgetting the higher $A_\infty$-structures. We prove that the $\infty$-contraction algebra recover the derived category of singularities
 $\D_{sg}(Y)$ and the triangulated subcategory $\cC^b\subset\D^b(coh X)$ of objects $F$ with $\bR f_{\ast} F=0$ (Theorem \ref{ReconCb}). 

In Section~\ref{sec:def}, 
we use the above alternative description of the 
contraction algebras to give 
their  
deformations 
to semisimple algebras with the 
numbers of blocks coincide with the sums of the 
GV invariants (Theorem~\ref{ss}).  
We are still not able to prove that our 
deformations of the contraction algebras are flat, 
but assuming this would give a different and more 
intrinsic proof of the dimension formula of the contraction 
algebra given in~\cite{To14}. 
We will also address the question regarding the 
reconstruction problem of 
GV invariants from the contraction algebras. 
Using the work of the second author on wall crossing formula of parabolic stable pairs \cite{Todpara}, we show that the
generating series of non-commutative Hilbert schemes associated with 
contraction algebras admit a product expansion whose power 
is given by GV invariants (Theorem~\ref{thm:GVformula}). 
This result also 
leads to an alternative proof of the dimension formula in \cite{To14}. 

In Section \ref{sec:inf}, we 
study Donovan-Wemyss's conjecture (Conjecture~\ref{GenDWconj}),
which says that the contraction algebra 
determines the formal neighborhood 
of the flopping curve. 
We will 
propose an $A_\infty$-enhancement 
(Conjecture \ref{GenDWconj}) of the 
above conjecture, 
using the $A_{\infty}$-structure 
of the contraction algebra. 
We prove it for the special case when the singularity 
of $Y$ is 
weighted homogeneous (Theorem \ref{WH}). 
Finally in Section \ref{sec:exam}, we study several examples of contraction algebras, including flopping contractions of type $cA_1$, $cD_4$.

\paragraph{Acknowledgments.} We are grateful to Will Donovan and Michael Wemyss for many valuable comments. The research of the first author is supported by RGC Early Career grant no. 27300214 and NSFC Science Fund for Young Scholars no. 11401501. The first author would like to thank to IPMU for its hospitality whilst this manuscript was being prepared.
The second author is supported by World Premier 
International Research Center Initiative
(WPI initiative), MEXT, Japan, 
and Grant-in Aid
for Scientific Research grant (No.~26287002)
from the Ministry of Education, Culture,
Sports, Science and Technology, Japan, 
and
JSPS Program for Advancing Strategic International Networks to Accelerate the Circulation of Talented Researchers.

\section{Preliminary}\label{sec:pre}
\subsection{Contraction algebra}
Let $X$ be a smooth quasi-projective complex 3-fold.  
A flopping contraction is a birational morphism 
\begin{equation}\label{flop}
f: X\to Y
\end{equation}
which is an isomorphism in codimension one, $Y$ has only Gorenstein singularities and the relative Picard number of $f$ equals to one. 
We denote by $C$ the exceptional locus of $f$, 
which is a tree of smooth rational curves. 
Below, we assume that $Y$ is an affine variety and 
the exceptional locus $C$ is isomorphic to $\mathbb{P}^1$. 
Let $\cL$ be a relative ample line bundle on $X$. Take a minimal generating set of $\Ext^1(\cO_X,\cL^{-1})$ 
as $\Gamma(\cO_X)$-module
consisting of $r$ elements. 
We define the vector bundle $\cN$ to be the extension:
\begin{equation}\label{extN}
\xymatrix{
0\ar[r] & \cL^{-1}\ar[r] & \cN\ar[r] &\cO_X^{\oplus r}\ar[r] &0
}
\end{equation}
We set $\cU:= \cO_X\oplus \cN$, $N:=\bR f_*\cN=f_*\cN$ and 
\[
A:= \End_{X}(\cU)\cong\End_{Y}(\cO_Y\oplus N).
\]
By Van den Bergh (\cite[Section~3.2.8]{VdB04}), $\cU$ is a tilting object. The functor $F:=\RHom_X(\cU,-)$ defines an equivalence of categories $\D^b(coh(X))\cong \D^b(\text{mod-}A)$. Obviously, $F(\cO_{X})$ and $F(\cN)$ are projective $A$-modules.

Let $p \in Y$ be the image of $C$ under $f$, 
and we set $R=\hat{\cO}_{Y,p}$ the formal completion of $\cO_Y$ at the singular point $p$. 
We take the completion of \ref{flop}
\begin{equation}\label{compflop}
\hat{f}:\hat{X}:=X\times_Y \Spec R\to \hat{Y} :=\Spec R.
\end{equation}
Let $\cL$ be a relative ample line bundle on $\hat{X}$ such that $\deg(\cL|_C)=1$. In this case, 
\[
A:= \End_{\hat{X}}(\cU)\cong \End_R(R\oplus N).
\]

The following proposition is due to Van den Bergh~\cite{VdB04}. 
\begin{prop}\emph{(\cite[Proposition~3.5.7]{VdB04})}\label{prop:simple}
Let $f: \hat{X}\to \hat{Y}$ be the flopping contraction as before.
Denote $C_l$ for the scheme theoretical fiber of $p\in Y$. Then $F(\cO_{C_l})$ and $F(\cO_C(-1)[1])$ are simple $A$-modules. 
\end{prop}

\begin{definition}(\cite[Definition~2.8]{DW13})  \label{A_con}
The contraction algebra $A_{con}$ is defined to be $A/I_{con}$, where $I_{con}$ is the two sided ideal of $A$ consisting of morphisms $R\oplus N\to R\oplus N$ factoring through a summand of finite sums of $R$.
\end{definition}
The contraction algebra $A_{con}$ is an augmented algebra. The augmentation morphism is defined by the algebra morphism $A\to \End_{A}(S)\cong \CC$, which factors through $A_{con}$.
Here $S=F(\cO_C(-1)[1])$ is the simple $A$-module in 
Proposition~\ref{prop:simple}. 
\begin{remark}
In \cite{DW13}, Donovan and Wemyss gave an alternative definition of the contraction algebra using noncommutative deformation functor (see Definition 2.11 \cite{DW13}). It is Morita equivalent with the contraction algebra in Definition \ref{A_con} but may not be isomorphic. In this paper, we will use only Definition \ref{A_con} (and its variation in Theorem \ref{Acon=End}). The Morita equivalence of the contraction algebras in these two definitions can be explained by Theorem \ref{ss}.
\end{remark}

\subsection{Categories associated to the singularity and its resolution}

Let $Y$ be a quasi-projective scheme. Denote $\D^b(coh(Y))$ for the bounded derived category of coherent sheaves on $Y$ and $\perf(Y)$ for the full subcategory consisting of perfect complexes on $Y$. We define a triangulated category $\D_{sg}(Y)$ as the quotient of $\D^b(coh(Y))$ by $\perf(Y)$ 
(cf.~\cite[Definition~1.8]{Orlov02}). Consider the case when $Y$ is a hypersurface in a smooth affine variety $\Spec B$ defined by $W=0$ for a function $W\in B$.
Orlov proved that the derived category of singularities $\D_{sg}(Y)$ is equivalent, as triangulated categories, with the homotopy category of the category of matrix factorizations $\text{MF}(W)$ (cf.~\cite[Theorem~3.9]{Orlov02}). The objects of $\text{MF}(W)$ are the ordered pairs:
\[\bar{E}=(E,\delta_E):=\xymatrix{
E_1\ar@/^1mm/[r]^{\delta_1} &E_0\ar@/^1mm/[l]^{\delta_0} 
}
\] where $E_0$ and $E_1$ are finitely generated projective $B$-modules  and the compositions $\delta_0\delta_1$ and $\delta_1\delta_0$ are the multiplications by the element $W\in B$. 
The precise definition of the category of matrix factorizations $\text{MF}(W)$ can be found in~\cite[Section~3]{Orlov02}.

Using the category of singularities, one can give an alternative definition for the contraction algebra $A_{con}$.  
First we recall a lemma of Orlov about  category of singularities for Gorenstein schemes.
\begin{prop}\emph{(\cite[Proposition~1.21]{Orlov02})}\label{DsgGor}
Let $Y$ be a Gorenstein scheme of finite Krull dimension. Let $\cF$ and $\cG$ be coherent sheaves on $Y$
such that $\sExt^i(\cF,\cO_Y)=0$ for all $i>0$. Fix 
$N$ such that $\Ext^i(\cP,\cG)=0$ for $i>N$ and for any locally free sheaf $\cP$. Then 
\[
\Hom_{\D_{sg}}(\cF,\cG[N])\cong \Ext^N(\cF,\cG)/ \cR
\] where $\cR$ is the subspace of elements factoring through locally free.
\end{prop}
The condition $\sExt^i(\cF,\cO_Y)=0$ is satisfied when $\cF$ is maximal Cohen-Macaulay. If we assume that $Y$ is affine then the integer $N$ can be chosen to be 0.

\begin{theorem}\label{Acon=End}
Let $f:X\to Y$ be a flopping contraction defined as above such that $Y=\Spec R$ is affine.
There is a natural isomorphism 
\[
A_{con}\cong \Hom^0_{\D_{sg}(Y)}(N,N).
\]
\end{theorem}
\begin{proof}
The quotient functor $\D^b(coh(Y))\to \D_{sg}(Y)$ induces a map from $A$ to $\Hom^0_{D_{sg}}(R\oplus N,R\oplus N)\cong \Hom^0_{D_{sg}}(N,N)$. By local duality
\[
\bR f_*\bR\sHom_{X}(\cN,f^!\cO_Y)\simeq \bR \sHom(f_*\cN, \cO_{Y}),
\]
$\Ext^{>0}_R(N,R)=0$. Proposition \ref{DsgGor} implies that for $i\geq 0$
\[
\Hom_{D_{sg}}(N,N[i])\cong \Ext^i_R(N,N) / \cR
\] where $\cR$ are morphisms from $N$ to $N$ that factors through a projective $R$ module. Take $i=0$ and the theorem is proved.
\end{proof}

Below, we assume that $Y$ is an affine hypersurface with one isolated singular point. 
By taking the completion at the singular point, we may further assume that $B=\CC[[x_1,\ldots,x_n]]$ and $W$ has an isolated critical point at the origin. We define the \emph{Milnor algebra} of $Y$ to be $B_W:=\CC[[x_1,\ldots,x_n]]/J_W$ with $J_W:=(\partial_1 W,\ldots, \partial_n W)$. In this case, it is well known that the Hochschild homology of $\text{MF}(W)$ is isomorphic, as $\ZZ/2$-graded vector spaces, with $B_W$ (for example, see~\cite[Section~2.4]{PV10}). This is in fact a consequence of the following theorem due to Dyckerhoff:
\begin{theorem}\emph{(\cite[Theorem~4.1]{Dyc09})}\label{Dyc}
Let $(B,W)$ be defined as before and consider the residue field $\CC$ as a $B/(W)$-module. Denote the corresponding matrix factorization by $k^{st}$. Then $k^{st}$ is a compact generator of the triangulated category of ${\rm{MF}}(W)$.
\end{theorem}

If $f:X\to Y$ is a projective morphism of quasi-projective varieties such that $\bR f_*\cO_X=\cO_Y$, then the functor 
\[
\bL f^*: \D(Y)\to \D(X)
\] embeds $\D(Y)$ as a right admissible triangulated subcategory of $\D(X)$. In particular, this holds for the three dimensional flopping contraction.
Define a subcategory $\cC$ of $\D(X)$ by 
\[
\cC=\{E\in\D(X): \bR f_*(E)=0\}.
\] This is the left orthogonal of $\bL f^*\D(Y)$. The objects in $\cC$ are support on the exceptional locus of $f$.

Denote $\cC^b$ for the subcategory of $\cC$ whose objects $E$ have bounded coherent cohomologies. 
In the case when $f: X\to Y$ is a three dimensional flopping contraction with the exceptional curve denoted by $C$, we have:
  
\begin{prop}
Any object $E\in \cC^b$ is a consequent extensions of shifts of $\cO_C(-1)$.
\end{prop}
\begin{proof}
This is an immediate consequence of Proposition~\ref{prop:simple}. 
\end{proof}
The
above proposition shows that
$\cC^b$ is generated by the compact object $\cO_C(-1)$. 
Let $\cD$ be a dg-category, $\ZZ$-graded or $\ZZ/2$-graded. If $G$ is a compact generator of $\cD$ then $\cD$ is quasi-equivalent with the category of dg-modules over the $\ZZ$-graded or $\ZZ/2$-graded dg-algebra $\Hom^\bullet_\cD(G,G)$. This is a consequence of the derived Morita theory of dg-categories. The $\ZZ/2$-graded variation of it can be found in~\cite[Chapter~5]{Dyc09}.  In the upcoming section, we will link the categories $\text{MF}(W)$ and  $\cC^b$ to the contraction algebra by choosing appropriate compact generators.

\section{Structures on the contraction algebra}\label{sec:struc}
Assume that $f:X\to Y$ is a three dimensional flopping contraction for $Y=\Spec R$ and $R=\CC[[x,y,z,w]]/(W)$. A coherent sheaf $E$ on $Y$ can be identified with a finitely generated $R$-module $M$. It represents an object in the category $\D_{sg}(Y)$, which is equivalent with $\text{MF}(W)$. We denote the corresponding matrix factorization of $M$ by $M^{st}$, called the \emph{stabilization} of $M$.

From the theorem of Dyckerhoff, we know that $k^{st}$ is a compact generator of $\text{MF}(W)$ for any isolated hypersurface singularities. Though $k^{st}$ contains all the information of the hypersurface singularity, it is usually hard to compute. For the singularity underlying a three dimensional flopping contraction, a different generator exists by the work of Iyama and Wemyss \cite{IW10}.
\begin{prop}\label{generator}
Let $\cU=\cN\oplus\cO_X$ be the tilting bundle on $X$ and $N$ be the $R$-module defined as $f_* \cN$.
Its stabilization $N^{st}$ is a compact generator of ${\rm{MF}}(W)$.
\end{prop}
\begin{proof}
Because $R\oplus N$ is a tilting object, proposition 5.10 of \cite{IW10} implies that $N$ is a generator. By~\cite[Section~4]{Dyc09}, $N^{st}$ is a compact object.
\end{proof}
For a general $R$-module $M$, the derived pull back $\bL f^*M$ is unbounded. However, for those modules that lie in the perverse heart in the sense 
of~\cite[Section~3]{VdB04}. We have the following vanishing properties for the derived pullbacks proved by Bodzenta and Bondal. 

\begin{theorem}\emph{(\cite[Lemma~3.5, 3.6]{BB15})}\label{L1v}
Suppose $f : X \to Y$ is a projective birational morphism of relative dimension one of quasi-projective Gorenstein varieties of dimension $n\geq 3$. The exceptional locus of $f$ is of codimension greater than one in $X$. Variety $Y$ has canonical hypersurface singularities of multiplicity two. Then $L^{-1}f^*f_*\cN=0$ for any object $\cN$ in the heart of $^{-1}\text{Per}(X/Y)$. 
\end{theorem}
Notice that the assumption of the above theorem is satisfied for the three dimensional flopping contraction.
\begin{corollary}\label{Ext1v}
Let $\cN$ and $N$ be defined in the previous section. Then 
we have $\Ext^1_{D_{sg}}(N,N)=0$ and 
\begin{align}\label{Acon:id}
\Hom^{\bullet}_{\D_{sg}}(N,N)=\Hom^0_{\D_{sg}}(N,N)=A_{con}. 
\end{align}
\end{corollary}
\begin{proof}
The result of Van den Bergh shows that the 
vanishing theorem in Theorem~\ref{L1v} is applicable to $\cN$.
By proposition \ref{DsgGor}, it suffices to show that $\Ext^1_R(N,N)$ vanishes. By adjunction, $\Ext^1_R(N,N)\cong \Ext^1_X(\bL f^*f_*\cN,\cN)$. The vanishing 
$\Ext^1_{D_{sg}}(N, N)=0$ 
follows from Theorem \ref{L1v} and the local to global spectral sequence.
The identities (\ref{Acon:id}) now follow from 
Theorem~\ref{Acon=End} and the above vanishing. 
\end{proof}
Applying the derived Morita theory of $\ZZ/2$-graded dg-category (c.f. ~\cite[Chapter~5]{Dyc09}), 
proposition \ref{generator} implies that $\text{MF}(W)$ is equivalent to the category of dg-modules over the $\ZZ/2$-graded
dg-algebra $\Hom^{\bullet}_{\text{MF}(W)}(N^{st},N^{st})\cong \Hom^{\bullet}_{\D_{sg}}(N,N)$. 
We denote $A_{con}^\infty$ for the minimal model of $\Hom^{\bullet}_{D_{sg}}(N,N)$ and call it the \emph{$\infty$-contraction algebra}. The contraction algebra $A_{con}$ is the algebra underlying $A^\infty_{con}$ forgetting the higher Massey products. In particular, $A_{con}^\infty$ and $A_{con}$ have the same underlying vector space. The $A_\infty$-structure on $A^{\infty}_{con}$ is always nontrivial unless the singularity is non-degenerated (which means the Hessian is non-degenerated). We will study this  
in Section~\ref{sec:inf} by one example. Moreover, the $A_\infty$-structure on $A_{con}^\infty$ is conjectured to reconstruct the local ring of the singularity
 (see conjecture~\ref{GenDWconj}).

In the rest of this section, we construct a Frobenius structure on $A_{con}$. 
Let $W$ be a function in $B:=\CC[[x_1,\ldots,x_n]]$ with isolated critical point at the origin.
Recall from~\cite[Section~2.4]{PV10} that
\[
\HH_\bullet(\text{MF}(W))\cong B_W\cdot d{\bf{x}}[n],
\] where $d{\bf{x}}=d x_1\wedge\ldots \wedge d x_n$.

Given a matrix factorization $\bar{E}=(E,\delta_E)$, its chern character $ch(\bar{E})\in \HH_0(\text{MF}(W))$ is defined to be
\[
ch(\bar{E})=\text{str}(\partial_n\cdot\ldots \partial_1)\cdot d{\bf{x}}\ \text{mod}\ J_W\cdot d{\bf{x}}
\] where $\text{str}$ is the supertrace of the matrix with entries in $B$.
There is a canonical bilinear form on $\HH_\bullet(\text{MF}(W))$ with the form:
\[
\langle f\otimes d{\bf{x}},g\otimes d{\bf{x}}\rangle = (-1)^{\frac{n(n-1)}{2}} \tr (f\cdot g),
\]
where $\tr$ is the trace on the Milnor algebra defined by the generalized residue:
\[
\tr(f)=\Res \left[\begin{array}{ccccc}f(x)\cdot dx_1\wedge\ldots \wedge dx_n\\
\partial_1 W,\ldots \partial_n W\end{array}\right].
\]
The Hirzebruch-Riemann-Roch formula for matrix factorization, proved by Polishchuk and Vaintrob (\cite{PV10}) says: 
\begin{equation}\label{HRR}
\chi(\bar{E},\bar{F})=\langle ch(\bar{E}),ch(\bar{F})\rangle.
\end{equation}

More generally, we can define a canonical ``boundary-bulk'' map
\[
\tau^{\bar{E}}: \Hom^\bullet(\bar{E},\bar{E})\to \HH_\bullet(\text{MF}(W))
\] 
by 
\[
\tau^{\bar{E}}(\alpha)=\text{str}(\partial_n \delta_E\cdot\ldots \partial_1\delta_E\circ \alpha)\cdot d{\bf{x}}\ {\rm{mod}}\ J_W\cdot d{\bf{x}}.
\]
It is clear that the chern character $ch(\bar{E})$ is a special case of $\tau^{\bar{E}}$ when $\alpha$ equals to the identity map of $\bar{E}$. 

By theorem \ref{Dyc}, $\text{MF}(W)$ is a saturated dg-category. A fundamental theorem of Auslander \cite{Aus78} states the Serre functor $\bf{S}$  on
$\text{MF}(W)$ is $(-)[n-2]$. Then there is a perfect pairing:
\[
\sigma: \Hom^i({\bf{S}}(\bar{E}),\bar{F})\otimes_\CC\Hom^i(\bar{F},\bar{E})\to\CC
\] for all $i$. Because $\text{MF}(W)$ is 2-periodic, the Serre functor is the identity functor if $n=4$. Apply to the case $\bar{E}=\bar{F}=N^{st}$, we have a perfect pairing: 
\[
\sigma: A_{con}\otimes_\CC A_{con}\to\CC
\]
In \cite{Mur13}, Murfet constructs $\sigma$ explicitly in term of the boundary bulk map. Apply Murfet's construction to our situation gives the following. Given $\alpha$ and $\beta$ in $A_{con}$, the pairing is 
\[
\sigma(\alpha,\beta)=\Res(\tau^{N^{st}}(\alpha\circ\beta)).
\]
The following proposition follows immediately. 
\begin{prop}\label{Frob}
The contraction algebra $A_{con}$ together with the bilinear form $\sigma$ forms a Frobenius algebra.
\end{prop}
The Milnor algebra $B_W$ is a commutative Frobenius algebra. The contraction can be viewed as a noncommutative analogue of $B_W$. We will show later that compare with $B_W$, the Frobenius algebra $A_{con}$ behaves better under deformations. Like $B_W$, the contraction algebra $A_{con}$ is also augmented. Given $\alpha\in A_{con}$, we call $\Res(\tau^{N^{st}}(\alpha))$ the \emph{trace} of $\alpha$. The following remark shows that there exists a canonical element (up to scalar) in $A_{con}$ with nonzero trace.
\begin{remark}\label{soc}
Denote the Jacobson ideal of $A_{con}$ by $J(A_{con})$. It is nothing but the kernel of the augmentation map. Denote $\text{soc}(A_{con})$ for the left ideal of $A$ (treated as a submodule of $_A A$) that is annihilated by $J(A_{con})$. By corollary 5.7 \cite{DW13}, $A_{con}$ is a self-injective algebra (therefore Auslander-Gorenstein). Denote the simple right $A_{con}$-module $F(\cO_C(-1)[1])$ by $S$. Apply $\Hom(-, A_{con})$ to the short exact sequence of right $A_{con}$-modules:
\[
\xymatrix{0\ar[r] &J(A_{con})\ar[r] &A_{con}\ar[r] &S\ar[r] &0 }
\]
By the self-injective property,  $\text{soc}(A_{con})$ is isomorphic to $\Hom_{A_{con}}(S,A_{con})$ as left $A_{con}$-modules. Therefore, $\text{soc}(A_{con})$ is one dimensional over $\CC$. Choose a generator $\alpha$ of $\text{soc}(A_{con})$. Because $\text{soc}(A_{con})$ is contained in all the nonzero ideals of $A_{con}$, the Frobenius property implies that $\Res(\tau^{N^{st}}(\alpha))$ is non zero. Moreover, $\tau^{N^{st}}(\text{soc}(A_{con}))=\text{soc}(B_W)$.
\end{remark}

\begin{prop}\label{commutator}
The boundary bulk map $\tau^{N^{st}}$ factors through $A_{con}/[A_{con},A_{con}]$.
\end{prop}
\begin{proof}
 By Lemma 1.1.3 and Lemma 1.2.4 of \cite{PV10}, $\tau^{N^{st}}$ can be interpreted as the trace of a natural transform from the functor $\RHom_{\text{MF(W)}}(N^{st},-)\otimes N^{st}$ to the identity functor. Therefore, it is zero on $[A_{con},A_{con}]$.
\end{proof}

To finish this section, we prove that the category $\cC^b$ defined in the previous section can be reconstructed from the contraction algebra. Let $S$ be the simple $A$-module defined in remark \ref{soc}. Because $\Hom_X(\cO_C(-1),\cO_X)=0$, $S$ is an $A_{con}$-module. The surjection $A\to A_{con}$ defines a functor $j_*: \text{mod-}A_{con}\to \text{mod-}A$. Under this notation, $j_*S$ is simply $S$ treated as an $A$-module.
Because $\cO_C(-1)$ is a compact generator of $\cC^b$, one can reconstruct $\cC^b$ from the $A_\infty$-algebra $\Ext^\bullet_X(\cO_C(-1),\cO_C(-1))$, which is isomorphic to $\Ext^\bullet_A(j_*S,j_*S)$ under the equivalence $F[1]$.
The reconstruction theorem of $\cC^b$ from $A_{con}$ can be proved by relating the deformation theory of $S$ and that of $j_*S$.
\begin{theorem}\label{ReconCb}
Suppose that $X\to Y$ and $X^\prime\to Y^\prime$ are three dimensional flopping contractions in smooth quasi-projective 3-folds, to points $p$ and $p^\prime$ 
respectively. Assume that $Y$ and $Y^\prime$ are spectrums of complete local rings. To these, associate the contraction algebras $A_{con}$ and $A^\prime_{con}$. Denote the category $\cC^b$ associated to $X$ and $X^\prime$ by $\cC^b(X)$ and $\cC^b(X^\prime)$. They are equivalent if $A_{con}\cong A^\prime_{con}$.
\end{theorem}
\begin{proof}
As $\cC^b$ is equivalent to the triangulated category of the category of modules over the $A_\infty$-algebra $\Ext_A^\bullet(j_*S,j_*S)$, it suffices 
to show this algebra (together with its $A_\infty$-structure) can be reconstructed from $A_{con}$. 

By the work of Keller \cite{Keller99}, the functor $j_*: \text{mod-}A_{con}\to \text{mod-}A$ admits a dg-lifting.  Therefore, it induces a morphism of $A_{\infty}$-algebras from $\Ext_{A_{con}}^\bullet(S,S)$ to  $\Ext_A^\bullet(j_*S,j_*S)$.  By adjunction, 
\[
\Ext_A^\bullet(j_*S,j_*S)\simeq \Ext_{A_{con}}^\bullet(\bL j^*j_*S,S).
\]
There is a spectral sequence with $E_2$ page 
\[
\Ext^{p}_{A_{con}}(\bL^{-q}j^*j_*S,S)
\] converge to $\Ext_A^{p+q}(j_*S,j_*S)$. By proposition 5.6 of \cite{DW13}, $A_{con}$ admits a projective resolution of $A$-modules of following
\[
\xymatrix{
0\ar[r] & P\ar[r] &Q_1\ar[r] &Q_0\ar[r] &P\ar[r] & A_{con}\ar[r] &0}
\] where $P=F(\cN)$, $Q_0$ and $Q_1$ are direct summand of sum of $F(\cO_X)$. 
The existence of the above resolution is a consequence of the Auslander-Gorenstein property of $A_{con}$.
So we have 
\begin{equation*}
\bL^{-q}j^*j_*S=\Tor_q^A(j_*S,A_{con})=
\begin{cases}
S & q=0,3\\
0 & q=1,2
\end{cases}.
\end{equation*}
Therefore, the $E_2$ page has nonvanishing terms exactly for $q=0,3$ where both rows are isomorphic to $\Ext_{A_{con}}^\bullet(S,S)$. The $E_4$ page of the spectral sequence has a nontrivial differential $d: \Ext_{A_{con}}^p(S,S)\to \Ext_{A_{con}}^{p+4}(S,S)$ for all $p\geq 0$. Because $\Ext_A^{n}(j_*S,j_*S)$ vanish for $n\geq4$, the differential $d$ are isomorphisms for $p>0$. For $p=0$, notice that both $\Ext_A^3(j_*S,j_*S)$ and $ \Ext_{A_{con}}^{3}(S,S)$ are one dimensional. Therefore, $d$ is an isomorphism for $p=0$.

In other words, if we denote $\Ext_{A_{con}}^\bullet(S,S)$ by $A^!_{con}$ and introduce a formal variable $h$ of degree 3 then we may construct an  $A_\infty$-algebra $A^!_{con}\oplus A^!_{con} h$ with $m_2,m_3,\ldots$ of $A^!_{con}$ extended linearly and  $m_1$ defined by 
\[
m_1(x+y\cdot h)=d(y).
\]
The $A_\infty$-algebra $\Ext_{A}^\bullet(j_*S,j_*S)$ is quasi-isomorphic to $A^!_{con}\oplus A^!_{con} h$ with the $A_\infty$-structure defined above. So $m_k$ on $\Ext_{A}^\bullet(j_*S,j_*S)$ are determined by $m_k$ on $\Ext_{A_{con}}^\bullet(S,S)$.  This proves that $\cC^b$ can be reconstructed from $A_{con}$.
\end{proof}
The above proof admits a geometric interpretation.
\begin{remark}
The Maurer-Cartan loci of the two $A_\infty$ algebras $\Ext_{A_{con}}^\bullet(S,S)$ are $\Ext_{A}^\bullet(j_*S,j_*S)$ are isomorphic. They have isomorphic obstruction theory up to $\Ext^3$. The contraction algebra $A_{con}$ is an example of singular Calabi-Yau algebra assuming the critical point of $W$ is degenerated. The module category of $A_{con}$ is an abelian subcategory of $\cC^b$ but its derived category is not equivalent to $\cC^b$. Because $\cC^b$ is saturated while $\D^b(\text{mod-}A_{con})$ is not.
\end{remark}


\section{Contraction algebra under deformation}\label{sec:def}
Let 
\[\xymatrix{\cal{X}\ar[r]^g\ar[rd] &\cal{Y}\ar[d]\\
&\cal{T}}\] be a flat deformation of the flopping contraction $f:X\to Y$, where $T$ is a Zariski open neighborhood of $0\in\AA^1$ such that $g_0=f$
and $g_t:\cX_t\to \cY_t$ for $t\neq 0$ is a flopping contraction. Denote the exceptional fiber of $g_t$ by $C_t$ and $C_0=C$ for $t=0$. By choosing a relative ample line bundle $\cL$ on $\cX$, the exact sequence (\ref{extN}) defines a vector bundle $\cN$ such that $\cO_{\cX}\oplus\cN$ is a tilting object.  For each $t\in\cT$, $N_t:=(g_t)_*\cN_t$ is a compact generator of $\D_{sg}(Y_t)$. 
As an analogy of Theorem~\ref{Acon=End}, 
we define 
$A_{con,t}$ to be 
\begin{align*}
A_{con, t} :=\Hom^0_{\D_{sg}(\cY_t)}(N_t,N_t).
\end{align*}
By theorem \ref{Ext1v}, the dimension of $A_{con,t}$ is equal to the Euler characteristic $\chi(N^{st}_t,N^{st}_t)$. 

Let $l$ be the length of the scheme theoretical exceptional fiber $C$
of $f$. 
The length $l$ takes values $1,2,\ldots,6$ with $l>1$ if and only if $C$ is a $(1,-3)$ curve. According to the classification theorem of Katz and Morrison \cite{KM92}, $Y$ must have compound du Val singularities of types $cA_1,cD_4,cE_6,cE_7$ or $cE_8$.

The most generic deformation of $f:X\to Y$ has the property that for $t\neq 0$, $\cY_t$ has finitely many $cA_1$ singularity and $\cX_t$ contains finitely many $(-1,-1)$ curves with possibly different homology classes contracting to those singularities. More specifically, if $C_0=C$ has length $l$ then the exceptional curves in $\cX_t$ can have homology classes $j[C_0]$ for $j=1,2,\ldots,l$. The number of the connected components of the exceptional fiber with class $j[C_0]$ is denoted by $n_j$. We will show that the dimension of $A_{con,t}$ stays constant under such a deformation.

The morphism $g$ is a flopping contraction, and admits a flop
\[
\xymatrix{\cX\ar@{-->}[rr]\ar[rd]_g&&\cX^+\ar[ld]^{g^+}\\
& \cY}
\]
such that we have the derived equivalence 
\[
\Phi^{\cO_{\cX\times_\cY \cX^+}}_{\cX\to\cX^+}: \D^b(\cX)\simeq\D^b(\cX^+).
\]
By composing the above equivalence twice, we obtain the autoequivalence 
\begin{equation}\label{flopflop}
\Phi^{\cO_{\cX\times_\cY \cX^+}}_{\cX^+\to\cX}\circ\Phi^{\cO_{\cX\times_\cY \cX^+}}_{\cX\to\cX^+}: \D^b(\cX)\simeq\D^b(\cX).
\end{equation}
Let $\Psi$ be an inverse of the equivalence (\ref{flopflop}), and 
\[
\cP\in \D^b(\cX\times_\cT \cX)
\]
be the kernel object of $\Psi$. In \cite{To14}, the second author proved the flatness of $\cH^1(\cP)$ over $\cT$, which leads to the following formula for the dimension of the the contraction algebra.
\begin{theorem}\emph{(\cite[Theorem~1.1]{To14})}\label{dimAcon}
We have the following formula
\[
{\rm{dim}}_\CC A_{con,0}=\sum_{j=1}^l j^2\cdot n_j,
\]
where $l$ is the scheme theoretical length of $C$.
\end{theorem}

The next result shows that the dimension of $A_{con,t}$ stays constant under a generic deformation.
\begin{theorem}\label{ss}
Let $g: \cX\to\cY$ be a generic deformation of the flopping contraction mentioned above. Choose a relative ample line bundle $\cL$ on $\cX$ such that $\text{deg}(\cL|_{C_0})=1$. Let $\cN$ be the vector bundle on $\cX$ defined by the exact sequence \ref{extN}. Define the contraction algebra $A_{con,t}$ to be $\Hom^0_{D_{sg}(\cY_t)}(N_t,N_t)$. For $t\neq0$, the contraction algebra $A_{con,t}$ is semi-simple and of the form
\[
A_{con,t}\cong \prod_{j=1}^l\prod_{k=1}^{n_j} \text{Mat}_{j}(\CC)
\] where $\text{Mat}_j(\CC)$ is the algebra of $j\times j$ complex matrices.
\end{theorem}
\begin{proof}
By the theorem of Van den Bergh \cite{VdB04}, $\cO_\cX\oplus \cN$ is a tilting bundle. Because $\cN$ is locally free, by base change $N_t:=(g_*\cN)_t$ is isomorphic to $(g_t)_*\cN_t$ and $\cO_{\cX_t}\oplus \cN_t$ is a tilting object for every $t$.

For $t\neq 0$, we denote the connected components of the exceptional fiber of $\cX_t\to \cY_t$ by $C_t^{j,k}$ where $j=1,2,\ldots,l$ and $k=1,2,\ldots,n_j$ for fixed $j$. The curve $C_t^{j,k}$ has homology class $j[C_0]$. Denote the formal completion of $\cX_t$ along $C_t^{j,k}$ by $\hat{X}_{t}^{j,k}$ and denote the restriction of $\cL$ to $\hat{X}_{t}^{j,k}$ by $L_t^{j,k}$.
Then $\text{deg}(L_t^{j,k})=j$ for all $k=1,\ldots,n_j$.

Fix $j$ and $k$, the line bundle $L_t^{j,k}$ is a $j$-fold tensor product of a line bundle $L$ on $\hat{X}_{t}^{j,k}$ such that $\text{deg}(L|_{C_t^{j,k}})=1$. In other words, we may write $L_t^{j,k}=L^j:=L^{\otimes j}$. We need to compute the tilting bundle corresponding to $L^j$. Since we are in the $(-1,-1)$ situation, $H^1(\hat{X}_{t}^{j,k},L^{-j})$ is generated by $H^1(C_t^{j,k},\cO(-j))\cong \CC^{j-1}$. We may choose its basis to be a minimal generating set of $H^1(\hat{X}_{t}^{j,k},L^{-j})$. The short exact sequence \ref{extN} defines the vector bundle $\cM^{j,k}_t$ on $\hat{X}_{t}^{j,k}$:
\[
\xymatrix{
0\ar[r] & L^{-j}\ar[r] & \cM^{j,k}_t\ar[r] &\cO_{\hat{X}_{t}^{j,k}}^{\oplus j-1}\ar[r] &0
}
\]
By the minimality, $\cM^{j,k}_t$ is isomorphic to $(L^{-1})^{\oplus j}$. 

Suppose that $m$ is the cardinality of the minimal generating set of $H^1(\cX,\cL^{-1})$. By the argument above, the restriction of $\cN$ to $\hat{X}_{t}^{j,k}$ is a direct sum $\cM^{j,k}_t\oplus \cO_{\hat{X}_{t}^{j,k}}^{\oplus m+1-j}$. The existence of the second factor is due to the fact that the restriction of the minimal generating set of $H^1(\cX,\cL^{-1})$ is not minimal as a generating set of $H^1(\hat{X}_{t}^{j,k},L^{-j})$.

Denote $\hat{Y}_t^{j,k}$ for the formal completion of of $\cY_t$ at $g_t(C_t^{j,k})$ and denote $M^{j,k}_t$ for $(g_t)_* \cM^{j,k}_t$.
The contraction algebra $A_{con,t}$ for $t\neq 0$ can be computed as
\[
\begin{split}
\Hom^0_{\D_{sg}(\cY_t)}(N_t,N_t)&\cong\prod_{j=1}^l\prod_{k=1}^{n_j} \Hom^0_{\D_{sg}(\hat{Y}_t^{j,k})}(M^{j,k}_t\oplus\cO_{\hat{Y}_t^{j,k}}^{\oplus m+1-j},M^{j,k}_t\oplus\cO_{\hat{Y}_t^{j,k}}^{\oplus m+1-j})\\
&=\prod_{j=1}^l\prod_{k=1}^{n_j} \Hom^0_{\D_{sg}(\hat{Y}_t^{j,k})}(M^{j,k}_t,M^{j,k}_t)\\
&=\prod_{j=1}^l\prod_{k=1}^{n_j} \text{Mat}_{j}(\CC) .
\end{split}
\] 
\end{proof}
The above theorem together with the dimension formula in theorem \ref{dimAcon} implies that $A_{con}:=\Hom^0_{\D_{sg}(\cY)}(N,N)$ is flat for a generic deformation $\cT$. We expect the flatness of the contraction algebra to hold in full generality.

\begin{conjecture}\label{flatness}
Let $\cX\to\cY\to\cT$ be an arbitrary flat deformation of flopping contraction. Denote $N$ for $g_*\cN$. The contraction algebra $A_{con}$ is flat over $\cT$. And the 0-th Hochschild (co)homology of the contraction algebra $A_{con}$ is flat over $\cT$. 
\end{conjecture}

By theorem \ref{ss}, the dimension of $\HH_0(A_{con,t})$ is
\[
\sum_{j=1}^l n_j.
\]
for $t\neq 0$. In all the known examples of $A_{con,0}$, we have verified that $\HH_0(A_{con,0})$ has dimension $\sum_{j=1}^l n_j$.
The conjecture above suggests the deformation invariance of the dimensions of the contraction algebra and its 0-th Hochschild (co)homology. Indeed, the positive integers $n_j$ have already appeared in the work of Bryan, Katz and Leung \cite{BKL01}. They coincide with the genus zero Gopakumar-Vafa (GV) invariants 
defined by Katz \cite{Katz08}.

We recall 
genus zero GV invariants 
associated to 
a 3-dimensional flopping contraction 
$f \colon X \to Y$
with $X$ smooth, 
$C$ a smooth $\mathbb{P}^1$. 
 For $j\in \mathbb{Z}_{\ge 1}$, let 
$M_{j}$ be the moduli space of 
one dimensional 
stable sheaves $F$ on $X$
satisfying 
$[F]=j[C]$ and $\chi(F)=1$. There is a 
symmetric perfect obstruction theory 
on $M_j$, and 
we define
\begin{align*}
n_j =\int_{[M_j]^{\rm{vir}}} 1 \in \mathbb{Z}. 
\end{align*}
Note that $M_j$ is topologically one point. 
There exist several other interpretations of 
$n_j$. 
\begin{enumerate}
\item $n_j$ coincides with the 
value of the Behrend function at the point of $M_j$. 

\item $n_j$ coincides with the dimension of the 
algebra $\cO_{M_j}$. 

\item Let $p\in S \subset Y$ be a general hypersurface,  
and $\overline{S} \subset X$ its proper transform. 
Let $I \subset \cO_{\overline{S}}$ be
the ideal sheaf of $C$ in $\overline{S}$, 
and $C^{(j)} \subset \overline{S}$ be the subscheme 
defined by the symboic power $I^{(j)}$. 
Then $n_j$ coincides with 
$\dim_{\mathbb{C}} \cO_{\mathrm{Hilb}(X), C^{(j)}}$. This is shown in \cite{Katz08}.

\item 
$n_j$ coincides with the number of $(-1, -1)$-curves 
with curve class $j[C]$ on a deformation 
of $X \to Y$. This is shown in~\cite[Section~2]{BKL01}.
\end{enumerate}
On the other hand
for $j\in \mathbb{Z}_{\ge 1}$, 
let $\mathrm{Hilb}_j(A_{con})$ be the moduli space of 
surjections $A_{con} \twoheadrightarrow M$
in $\text{mod-}A_{con}$
with $\dim_{\mathbb{C}}M=j$. 
We define
\begin{align*}
\mathrm{DT}_j(A_{con})=
\int_{\mathrm{Hilb}_j(A_{con})} \nu \cdot d\chi. 
\end{align*}
Here $\nu$ is the Behrend function on $\mathrm{Hilb}_j(A_{con})$. 
\begin{theorem}\label{thm:GVformula}
We have the following formula
\begin{align}\label{eq:DTAcon}
1+
\sum_{j>0}
\mathrm{DT}_j(A_{con}) t^j
=\prod_{j\ge 1}\left(1-(-1)^j t^j\right)^{j n_j}. 
\end{align}
\end{theorem}
\begin{proof}
Let $\cC^b \subset \D^b (coh(X))$ be the subcategory consisting of
$E \in \D^b (coh(X))$ with $\bR f_{\ast}E=0$. 
The Van-den-Bergh's equivalence
restricts to the equivalence
\begin{align*}
\cC_0=\cC^b \cap coh(X) \stackrel{\sim}{\to}
\text{mod-}A_{con}. 
\end{align*}
Note that $\cC_0$ is the extension closure of $\cO_C(-1)$. 
In particular, any object in $\cC_0$ is 
a semistable sheaf with the same slope. 

Now we fix a divisor $H \subset X$ 
which intersects with $C$ at one point
transversally. 
We recall the notion of parabolic stable pairs introduced in~\cite{Todpara}, 
applied in this situation. 
By definition, a $H$-parabolic stable pair consists 
of a pair $(F, s)$, where 
$F \in \cC_0$ and $s \in F \otimes \cO_H$, 
satisfying the following:
for any surjection $ \pi \colon F \twoheadrightarrow F'$ in 
$\cC_0$
with $F' \neq 0$, we have 
$(\pi \otimes \cO_H)(s) \neq 0$. 
Let $M_j^{\rm{par}}$ be the moduli 
space of parabolic stable pairs $(F, s)$ 
with $[F]=j[C]$, which exists as 
a projective scheme by~\cite{Todpara}. 
 We show that 
there is an isomorphism of schemes
\begin{align}\label{isom:par}
M_j^{\rm{par}} \stackrel{\cong}{\to}
\mathrm{Hilb}_j(A_{con}). 
\end{align}
Indeed, we have the commutative diagram
\begin{align*}
\xymatrix{  
\cC_0 \ar[rr]^{\Phi} \ar[rd]_{\otimes \cO_H} &  & \text{mod-}A_{con} 
\ar[ld]^{\rm{forg}} \\
&  \mathbb{C}\text{-Vect}  &
}
\end{align*}
Hence giving $s \in F \otimes \cO_H$ for $F \in \cC_0$
is equivalent to giving 
an element of $\tau \in \Phi(F)$. 
Moreover $(F, s)$ is parabolic stable 
if and only if $\tau \in \Phi(F)$
generates $\Phi(F)$ as a right $A_{con}$-module. 
Therefore we obtain the isomorphism (\ref{isom:par}). 

We define $\mathrm{DT}_{j}^{\rm{par}} \in \mathbb{Z}$
to be
\begin{align*}
\mathrm{DT}_{j}^{\rm{par}} =
\int_{M_j^{\rm{par}}} \nu' \cdot d\chi. 
\end{align*}
Here $\nu'$ is the Behrend function on $M_j^{\rm{par}}$. 
By ~\cite[Proposition~3.16]{TodS}, we 
have the formula
\begin{align*}
1+
\sum_{j>0}
\mathrm{DT}_j^{\rm{par}} t^j
=\prod_{j\ge 1}\left(1-(-1)^j t^j\right)^{j n_j}. 
\end{align*}
By the isomorphism (\ref{isom:par}), 
we have $\mathrm{DT}_{j}^{\rm{par}}=\mathrm{DT}_{j}(A_{con})$, 
hence obtain the desired result. 
\end{proof}
As a corollary, we reconstruct 
the result of~\cite{To14}. 
\begin{corollary}\label{cor:dim}(Theorem \ref{dimAcon})
We have the formula
\begin{align*}
\dim_{\mathbb{C}}A_{con}=\sum_{j=1}^{l} j^2 n_j. 
\end{align*}
\end{corollary}
\begin{proof}
Since $\mathrm{Hilb}_{j}(A_{con})=\emptyset$
for $j>\dim_{\mathbb{C}}A_{con}$ and 
coincides with $\Spec \mathbb{C}$
for $j=\dim_{\mathbb{C}}A_{con}$, 
the equation (\ref{eq:DTAcon})
is written as
\begin{align*}
t^{\dim_{\mathbb{C}}A_{con}}+O(t^{\dim_{\mathbb{C}}A_{con}-1})
=t^{\sum_{j=1}^{l} j^2 n_j}+O(t^{\sum_{j=1}^{l} j^2 n_j-1}). 
\end{align*}
\end{proof}
We can also characterize $n_j$ and $l$ from the contraction algebra. 
\begin{corollary}
Let $f' \colon X' \to Y'$ be another 
3-fold flopping contraction with $X'$ smooth and 
the exceptional curve $C'$ being a smooth $\mathbb{P}^1$. 
Let $A_{con}'$, $l'$ and $n_j'$ be the corresponding
contraction algebra, length and the GV invariants. 
If $A_{con} \cong A_{con}'$ as algebras, we 
have $n_j=n_j'$ for all $j\ge 1$. 
In particular, we have $l=l'$. 
\end{corollary}
\begin{proof}
The formula (\ref{eq:DTAcon}) gives
\begin{align*}
\prod_{j\ge 1}\left(1-(-1)^j t^j\right)^{j n_j}
=\prod_{j\ge 1}\left(1-(-1)^j t^j\right)^{j n_j'}.
\end{align*}
By comparing the coefficients of the above equality, 
we easily obtain $n_j=n_j'$ for all $j\ge 1$. 
Since we have
\begin{align*}
l= \mathrm{max}\{j\ge 1 : n_j \neq 0\}, 
\end{align*}
we in particular obtain $l=l'$. 
\end{proof}

\section{The $\infty$-contraction algebra}\label{sec:inf}
In this section, we study the $A_\infty$-structure on $A^\infty_{con}$. When the singularity is weighted homogeneous, the Milnor algebra completely determine the local ring of the singularity. In the general cases, the Milnor algebra is not sufficient to determine the local ring. The additional structure needed is the action of the defining equation of the hypersurface on the Milnor algebra.

Let $A$ and $B$ be two $\ZZ$-graded or $\ZZ/2$-graded $A_\infty$ algebras. We say $B$ is \emph{derived Morita equivalent} with $A$ if there exists an $A_\infty$ $A-B$-bimodule $X$ such that the functor $-\otimes_A^L X$ induces a derived equivalence 
\[
\D A\cong \D B
\] where $\D A$ and $\D B$ are the derived categories of modules over $A$ and $B$. 
The following proposition shows that the $\infty$-contraction algebra recovers the Milnor algebra.
\begin{prop}\label{Recon1}
Suppose that $\hat{X}\to \hat{Y}$ and $\hat{X}^\prime\to \hat{Y}^\prime$ are three dimensional flopping contractions such that that $\hat{Y}$and $\hat{Y}^\prime$ are the spectrum of $\CC[[x,y,z,w]]/(W)$ and $\CC[[x,y,z,w]]/(W^\prime)$ respectively, for germs of functions $W$ and $W^\prime$. The Milnor algebras $B_W$ and $B_{W^\prime}$ are isomorphic as algebras if the corresponding $\infty$-contraction algebras are derived Morita equivalent.
\end{prop}
\begin{proof}
By the derived Morita equivalence of $\ZZ/2$-graded dg-category and proposition \ref{generator}, $\D_{sg}(\hat{Y})$ is equivalent with the triangulated category of the dg category of modules over $A_{con}^\infty$. Therefore, the Hochschild cohomology of $\text{MF(W)}\cong\D_{sg}(\hat{Y})$ is isomorphic to the Hochschild cohomology of $A^\infty_{con}$ as algebras. By corollary 6.4 of \cite{Dyc09}, 
$\HH^\bullet(\text{MF}(W))$ is isomorphic to the Milnor algebra $B_W$ as algebras.
\end{proof}

In \cite{DW13}, Donovan and Weymss conjecture that the local ring of $Y$ at the singularity $p$ is determined by the contraction algebra.
\begin{conjecture}\emph{(\cite[Conjecture~1.4]{DW13})}\label{conjDW}
 Suppose that $X\to Y$ and $X^\prime\to Y^\prime$ are three dimensional flopping contractions in smooth quasi-projective 3-folds, to points $p$ and $p^\prime$ respectively.   To these, associate the contraction algebras $A_{con}$ and $B_{con}$. Then the completions of stalks at $p$ and $p^\prime$ are isomorphic if and 
only if $A_{con}\cong B_{con}$.
\end{conjecture}
From theorem \ref{ss}, we see that the different choices of the ample line bundles may lead to Morita equivalent contraction algebras. So it is natural to replace the condition $A_{con}\cong B_{con}$ in the above conjecture by the condition that $A_{con}$ and $B_{con}$ being Morita equivalent. However, this may not be sufficient. In section~\ref{sec:exam},  we will see that for almost all examples the $A_\infty$-structure on the $\infty$-contraction algebra is nontrivial.

We propose the following variation of the conjecture of Donovan and Wemyss and prove one special case.
\begin{conjecture}\label{GenDWconj}
 Suppose that $X\to Y$ and $X^\prime\to Y^\prime$ are three dimensional flopping contractions in smooth quasi-projective 3-folds, to points $p$ and $p^\prime$ respectively.   To these, associate the $\infty$-contraction algebras $A^\infty_{con}$ and $B^\infty_{con}$. Then the completions of stalks at $p$ and $q$ are isomorphic if and 
only if $ A^\infty_{con}$ and $ B^\infty_{con}$ are derived  Morita equivalent.
\end{conjecture}
Let $f$ and $g$ be two germs of holomorphic (or formal) functions on $(\CC^n,0)$. We say $f$ and $g$ are \emph{right equivalent} if they differ by a holomorphic (or formal) change of variables on the domain. We say $f$ is \emph{weighted homogeneous} if $f$ is right equivalent with a weighted homogeneous polynomial.  In section~\ref{sec:exam}, we give several examples of compound du Val singularities that admits small resolutions. Among these examples, the type $cA_1$ singularities and the Laufer's $cD_4$ singularity are weighted homogeneous. 

The following result is a consequence of the famous theorem of Mather and Yau.
\begin{theorem}(Proposition 2.5, Theorem 4.2 \cite{BY90})\label{MY}
Let $f,g:(\CC^n,0)\to\CC$ be two germs of holomorphic functions with isolated singularities at the origin. Denote the corresponding completions of the rings of functions on the hypersurfaces at the origin by $\widehat{\cO/(f)}$ and  $\widehat{\cO/(g)}$ and denote the corresponding Milnor algebras by $B_f$ and $B_g$. Assume that $f$ and $g$ are weighted homogeneous. Then  $\widehat{\cO/(f)}$ is isomorphic to $\widehat{\cO/(g)}$ if and only if $B_f\cong B_g$ as algebras.
\end{theorem}
\begin{theorem}\label{WH}
Assume that $Y$ and $Y^\prime$ have weighted homogeneous hypersurface singularities. Conjecture \ref{GenDWconj} holds.
\end{theorem}
\begin{proof}
By theorem \ref{MY}, the local rings of the weighted homogeneous isolated hypersurface singularities are isomorphic if and only if their Milnor algebras are isomorphic.  Then by proposition \ref{Recon1}, the Milnor algebra can be reconstructed from the infinity contraction algebra by taking its Hochschild cohomology. So the conjecture is proved.
\end{proof}

To finish this section, we sketch a possible way to prove conjecture~\ref{GenDWconj} for the general cases. And we will return to this in the future work.

The Milnor algebra $B_f$ is equipped with a $\CC[t]$-algebra structure by letting $t$ acts by multiplying with $f$, treated as an element in $B_f$. This action is trivial if $f$ is weighted homogeneous since $f$ belongs to the jacobian ideal $J_W$. However, it is not the case in general (see the $cE_6$ singularity example in section \ref{sec:exam}). Though the $\CC$-algebra structure on $B_f$ does not determine the local ring of the singularity in general (see the counter example in \cite{BY90}),  the $\CC[t]$-algebra structure on $B_f$ determines $\widehat{\cO/(f)}$ (cf.~\cite[Theorem~2.28]{GLS06}). 
The conjecture~\ref{GenDWconj}  can be reduced to the problem of finding a categorical construction for the action of $f$ on $B_f$. 

Let $A$ be a $\ZZ$-graded or $\ZZ/2$-graded $A_\infty$-algebra. 
 Then Hochschild cochain complex
\[
C^\bullet(A,A):=\Hom^\bullet(\bigoplus_{n\geq0} A[1]^{\otimes n},A)
\] 
is a $B_\infty$-algebra in the sense of Getzler and Jones (section 5.2 \cite{GJ94}). The $B_\infty$-structure induces the cup product and Gerstenhaber bracket on the Hochschild cohomology.
Keller proved that the derived Morita equivalences preserve the homotopy type of $C^\bullet(A,A)$ as $B_\infty$-algebras.
\begin{theorem}\label{invHH}\emph{(\cite[Section~3.2]{Keller03})}
Let $A$ and $B$ be dg-algebras and let $X$ be a dg $A-B$-bimodule. Suppose the functor 
\[
-\otimes_A^L X: \D A\to\D B
\] is an equivalence, there is a canonical isomorphism 
\[
\phi_X : C^\bullet(B,B)\to C^\bullet(A,A)
\] in the homotopy category of $B_\infty$-algebras.
\end{theorem}
Take $B$ to be the $\infty$-contraction algebra $A^\infty_{con}$. We speculate that the $W$ can be lifted to an even cocycle in $C^\bullet(B,B)$ using the $B_\infty$-structure such that its cohomology class is invariant under the derived Morita equivalence.

\section{Examples}\label{sec:exam}
In this section, we study the ($\infty$-)contraction algebras associated to three types of singularities. The first class of examples are compound du Val singularities of type $cA_1$ that admit small resolution by a length 1 curve. They are classified by a positive integer $k$ called width. The normal bundle of the exceptional curve is $\cO(-1)\oplus\cO(-1)$ if $k=1$ and $\cO\oplus\cO(-2)$ when $k>1$. In the $k=1$ case, the contraction algebra has trivial $A_\infty$-structure. However, for $k>1$ the contraction algebra carries nontrivial $A_\infty$ structure. 

The second example have singularities of type $cD_4$. Its scheme theoretical fiber has length 2. We calculate the dimension of the contraction algebra using the Riemann-Roch formula of the matrix factorization. This calculation requires no knowledge about the ring structure on $A_{con}$. Moreover, we compute the boundary bulk map and the Frobenius structure explicitly and verify conjecture \ref{flatness}. 

In \cite{Wey}, Weymss and Brown compute the contraction algebra of a $cE_6$ flopping contraction. In this example, the fiber has length 3. This is a first non-homogeneous example. Based on the result of \cite{Wey}, we compute the boundary bulk map for it. It turns out that conjecture \ref{flatness} also holds. 

\subsection{Length 1 case}

Consider the hypersurface singularity 
\[
W=x^2-y^{2k}+zw=0, 
\] where $k$ is a positive integer. Denote $R$ for $\CC[[x,y,z,w]]/(W)$ and $Y$ for $\Spec R$. It admits a $2\times 2$ matrix factorization
\[
\Psi=\left(\begin{array}{cc}w & -x-y^k\\x-y^k & z\end{array}\right), \ \Phi=\left(\begin{array}{cc}z& x+y^k\\-x+y^k& w\end{array}\right).
\]
Set $N=\text{coker}(\Psi)$. As usual, we denote the above matrix factorization by $N^{st}$. $Y$ admits two small resolutions $X$ and $X^+$ such that the exceptional fibers are both a  smooth rational curve of length 1. Denote the exceptional curve of $f: X\to Y$ by $C$. The normal bundle of $C$ is $\cO(-1)\oplus \cO(-1)$ if $k=1$ and $\cO\oplus\cO(-2)$ if $k>1$. The number $k$ is referred as the \emph{width}. By the theorem of Van den Bergh \cite{VdB04}, $R\oplus N$ is a tilting object for $\D^b(coh(X))$ (cf.~\cite[Section~3.3]{AM10}). The scalar multiplication by $y$ defines an endomorphism of $N^{st}$, and therefore an element in $A_{con}$.  The dimension of $A_{con}$ is computed by the Riemann-Roch formula of matrix factorization (\ref{HRR}).  We obtain $\chi(N^{st},N^{st})=k$.
One can verify that $A_{con}$ is isomorphic to $\CC[a]/(a^k)$ with $a$ acts by scalar multiplication with $y$. 
The Milnor algebra $B_W$ is $\CC[y]/(y^{2k-1})$.  The boundary bulk map $\tau$ sends $a^n$ to $y^{k-1+n}$ for $n=0,1,\ldots,k-1$. The socle ideal $(a^{k-1})$ is mapped to the socle ideal $(y^{2k-1})$. The kernel of the Frobenius trace on $A_{con}$ is $\CC\{1,a,\ldots,a^{k-2}\}$.

Consider the case $k=2$. The contraction algebra carries nontrivial $A_\infty$ structure. 
By the calculation above, $A_{con}$ is isomorphic to the algebra of dual number $\CC[\ep]/\ep^2$ and the  $B_W$ is $\CC[y]/y^3$. 
The bar complex $Bar(\CC[\ep]/\ep^2)$ is 
\[
\xymatrix{
0&\CC\ar[l]&\CC u^{-1}\ar[l]_0&\CC u^{-2}\ar[l]_0&\ar[l]&\ldots
}
\] with $\deg(u)=1$.
The normalized Hochschild cochain $C^n=\Hom(\CC u^n,\CC[\ep]/\ep^2)$ can be identified with $\CC[\ep]/\ep^2$ by $u^n\mapsto a+b\ep$. The differential $Q$ acts by
\[
\begin{split}
Q_0&=0 \\
Q_n(a+b\ep)&=(1+(-1)^{n+1})a\ep, \ \ \ n>1.
\end{split}
\]
As a consequence, 
\[
\HH^i(\CC[\ep]/\ep^2)=
\begin{cases}
\CC[\ep]/\ep^2 & i=0;\\
\CC & i=2k, k>0;\\
\CC & i=2k+1, k\geq0
\end{cases}
\]
On the other hand, because $N^{st}$ is a compact generator of $\text{MF}(W)$  the Hochschild cohomology of $A^\infty_{con}$ is isomorphic to the Milnor algebra $B_W=\CC[w]/w^3$. 
Using the method developed in \cite{AM10}, we compute the $A_\infty$ structure on $A^\infty_{con}$. It turns out that there is a higher Massey product 
\[
m_4(\ep,\ep,\ep,\ep)=1.
\]
In this example, $n_1=2$ and $n_j=0$ for $j>0$. The dimension of the Milnor algebra is 3. Under a generic deformation, $W=0$ is deformed to two $cA_1$ singularities of $k=1$. So we see that the dimension of $\HH_0(A_{con})$ is a deformation invariant while the dimension of the Milnor algebra is not.

\subsection{Laufer's flop}
The following example, called Laufer's flop, is studied
 in~\cite[Section~4.4]{AM10} and~\cite{DW13}. 

Consider the hypersurface singularity
\[
W=x^2+y^3+wz^2+w^{2k+1}y=0
\] for $k>0$. Notice that it is a weighted homogeneous hypersurface singularity with weight 
\[
(\frac{6k+3}{4},\frac{2k+1}{2},\frac{6k+1}{4},1).
\]
It admits a matrix factorization:
\[
\Psi=\left(\begin{array}{cccc}x & y &z&w^k\\-y^2&x&-yw^k&z\\
-wz&w^{k+1}&x&-y\\ -yw^{k+1}&-wz&y^2&x \end{array}\right), \ \Phi=\left(\begin{array}{cccc}x&-y&-z&-w^k\\y^2&x&yw^k&-z\\wz&-w^{k+1}&x&y\\yw^{k+1}&wz&-y^2&x\end{array}\right).
\]
Denote $\text{coker}(\Psi)$ by $N$. One can check that $R\oplus N$ is a tilting object.
We calculate that $ch(N^{st})= 2(6k+3) yw^k$. To compute the residue, we need to make a change of variable 
\[
\left(\begin{array}{c}2x\\3y^3\\ \frac{2z^3}{2k+1}\\ \frac{(2k+1)w^{4k+2}}{3}\end{array}\right)=\left(\begin{array}{cccc}1&0&0&0\\0&y&\frac{z}{2(2k+1)}&\frac{-w}{2k+1}\\0&0&-w^{2k-1}y&\frac{2z}{2k+1}\\ 0&\frac{(2k+1)w^{2k+1}}{3}&\frac{yz}{2}&-yw\end{array}\right)\cdot \left(\begin{array}{c}\partial_x W\\ \partial_y W\\ \partial_z W\\ \partial_w W\end{array}\right)
\]
The determinant of the above matrix is 
\[
y(y^2w^{2k}-\frac{yz^2}{2k+1})+\frac{2k+1}{3}w^{2k+1}(\frac{z^2}{(2k+1)^2}-\frac{w^{2k}y}{2k+1}).
\]
\begin{equation*}
\begin{split}
\chi(N^{st},N^{st})&=\Res \left[\begin{array}{ccccc}4(6k+3)^2 y^2 w^{2k}\cdot dx_1\wedge\ldots \wedge dx_n\\
\partial_1 W,\ldots \partial_n W\end{array}\right]\\
&=\Res \frac{4(6k+3)^2y^2w^{2k}\cdot \frac{w^{2k+1}z^2}{6k+3}}{2x\cdot 3y^2\cdot \frac{2z^2}{2k+1}\cdot \frac{2k+1}{3}w^{4k+2}}\\
&=(6k+3)\Res \frac{1}{xyzw}\\
&=6k+3
\end{split}
\end{equation*}
So $A_{con}$ has dimension $6k+3$. Using the method in \cite{AM10}, Donovan and Wemyss computed the algebra structure and obtain 
\[
A_{con}\cong\CC\langle a,b\rangle / (ab+ba,a^2-b^{2k+1}).
\] (See~\cite[Example~3.14]{DW13}.) 

We focus on the simplest case of $k=1$. 
The generators $a$ and $b$ are degree zero morphisms of $N^{st}$. They  are  represented by the matrices 
\[
a=\left(\begin{array}{cccc}0&1&0&0\\-y&0&0&0\\0&0&0&1\\0&0&-y&0\end{array}\right) \ \
b=\left(\begin{array}{cccc}0&0&1&0\\0&0&0&-1\\-w&0&0&0\\0&w&0&0\end{array}\right) 
\]  (cf.~\cite[Section~4.4]{AM10}.)

The Milnor algebra $B_W$ is a vector space of dimension 11 with basis
\[
1,y,y^2,y^2w,yz,yz^2,yw,z,z^2,w,w^2.
\]
The socle ideal is spanned by the element $yz^2=-3y^2w^2$
We compute the boundary bulk map from $A_{con}$ to $B_W$
\begin{equation*}
\begin{array}{cccc}
\tau(1)= -18 yw, & \tau(a)= 6yz, & \tau(b)=6w^3\\
\tau(a^2)=18y^2w, & \tau(a^2b)=-6yw^3=0, &\tau(b^2)=18yw^2,  &\tau(a^2b^2)=18y^2w^2.
\end{array}
\end{equation*}
The commutator $[A_{con},A_{con}]$ consists of elements $a^2b,ab$ and $ab^2$. By proposition \ref{commutator}, the boundary bulk morphism will map them to zero. The socle ideal $(a^2b^2)$ of $A_{con}$ maps to the socle ideal of $B_W$.

In this example, $n_1=5$, $n_2=1$ and $n_j=0$ for $j>0$. The first Hochschild homology $\HH_0(A_{con})$ is $A_{con}/[A_{con},A_{con}]$. As a vector space, it is spanned by $1,a,a^2,b,b^2,a^2b^2$. So the dimension is  $n_1+n_2$.


\end{document}